\documentclass[11pt]{amsart}   	
\usepackage[margin=1in]{geometry}                		%
\usepackage{cite}
\usepackage{amssymb,amsthm,amsmath,mathrsfs}
\usepackage[utf8]{inputenc}
\usepackage[english]{babel}
\usepackage{fancyhdr,color}
\usepackage{amsmath,amssymb,amsthm,graphicx,enumerate, fancyhdr, microtype}
\usepackage{amssymb, amsmath,amsthm,xcolor,bm,hyperref,color,tikz,euscript,verbatim}
\usepackage{xcolor,bm,hyperref,color,tikz,euscript,sansmath,ytableau, mathtools}
\usetikzlibrary{3d}
\usetikzlibrary{positioning,matrix,arrows,decorations.pathmorphing}

\usetikzlibrary{positioning,arrows}
\usetikzlibrary{decorations.markings}
\tikzstyle{chamberlabel} = [draw, teal!60!gray, shape=circle, minimum size=5pt, inner sep=1pt]
\tikzstyle{invisivertex} = [black, shape=rectangle, minimum size=0pt, inner sep=2pt]
\tikzstyle{point}=[draw, black, fill,shape=circle, minimum size=4pt, inner sep=0pt]
\tikzstyle{moebius} = [draw, teal, shape=circle, minimum size=8pt, inner sep=1.5pt]

\newcommand{\A}{\mathcal{A}}
\newcommand{\K}{\mathcal{K}}
\newcommand{\LL}{\mathcal{L}}
\newcommand{\LI}{\mathcal{L}^{\text{int}}}
\newcommand{\Z}{\mathbb{Z}}
\newcommand{\Q}{\mathbb{Q}}
\newcommand{\C}{\mathcal{C}}
\newcommand{\R}{\mathbb{R}}

\newcommand{\rank}{\mathsf{rank}}
\newcommand{\kk}{\mathbb{F}}

\newcommand{\GGG}{{\mathcal G}}
\newcommand{\NNN}{{\mathcal N}}
\newcommand{\CCC}{{\mathfrak C}}
\newcommand{\DDD}{{\mathfrak D}}

\newcommand{\grr}{{\mathfrak{gr}}}
\newcommand{\vgrr}{{\mathcal{V}}}

\newcommand{\init}{\mathrm{in}}
\newcommand{\Hilb}{\mathrm{Hilb}}
\newcommand{\Poin}{\mathrm{Poin}}

\newcommand{\galen}[1]{\textcolor{violet}{[#1]}}
\newcommand{\changes}[1]{\textcolor{black}{#1}}

\newtheorem{lemma}{Lemma}

\newtheorem{thm}[lemma]{Theorem}
\newtheorem*{coro}{Corollary}
\newtheorem{prop}[lemma]{Proposition}

\theoremstyle{definition}
\newtheorem{defn}[lemma]{Definition}

\theoremstyle{remark}
\newtheorem{observation}[lemma]{Observation}
\newtheorem{example}[lemma]{Example}
\newtheorem*{non-numbered-example}{Example}
\newtheorem{rem}[lemma]{Remark}
\newtheorem{question}[lemma]{Question}
\newtheorem{remark}[lemma]{Remark}

\title{The Varchenko-Gel'fand Ring of a Cone}
\author{Galen Dorpalen-Barry}
\address{Fakultät für Mathematik, Ruhr-Universit\"at Bochum, Germany}
\email{galen.dorpalen-barry@rub.de}
\thanks{The author was supported by NSF grant DMS-1601961.}
\date{\today}



\begin{document}

\maketitle

\begin{abstract}
For a hyperplane arrangement in a real vector space, the coefficients of its Poincaré polynomial have many interpretations.  An interesting one is provided by the Varchenko-Gel'fand ring, which is the ring of functions from the chambers of the arrangement to the integers with pointwise addition and multiplication. Varchenko and Gel'fand gave a simple presentation for this ring, along with a filtration and associated graded ring whose Hilbert series is the Poincaré polynomial.
We generalize these results to cones defined by intersections of halfspaces of some of the hyperplanes and prove a novel result for the Varchenko-Gel'fand ring of an arrangement: when the arrangement is supersolvable the associated graded ring of the arrangement is Koszul.
\end{abstract}


\section{Introduction}

Let $\A$ be a central hyperplane arrangement in $V\cong\R^\ell$, \changes{i.e. a finite collection of distinct, codimension one subspaces of $V$.}
A \emph{chamber} of $\A$ is a connected component of $V\backslash \bigcup_{H\in \A} H$ and we use $\C(\A)$ to denote the set of chambers of $\A$.
This paper concerns the Varchenko-Gel'fand\footnote{In Russian, Varchenko comes first (alphabetically speaking).} ring $VG(\A)$ of $\A$, a ring consisting of maps $\C(\A)\to \Z$ with pointwise addition and multiplication.
This ring was first introduced by Gel'fand and Varchenko \cite{VarchenkoGelfand}, who proved that it has a $\Z$-basis of monomials indexed by \emph{no broken circuit sets} of $\A$ and showed that the \emph{degree filtration} of $VG(\A)$ yields an \emph{associated graded ring} with Hilbert series
completely determined by $\LL(\A)$.
Since then, Cordovil \cite{Cordovil}, Gel'fand-Rybnikov \cite{GelfandRybnikov}, Moseley \cite{moseley}, Proudfoot \cite{proudfoot}, and others
have studied the Varchenko-Gel'fand ring of an arrangement, and various generalizations.

The main result of this paper, Theorem \ref{thm:main theorem}, extends their work to \emph{cones}, which are intersections of open halfspaces defined by some of the hyperplanes of $\A$.
Cones are interesting because they connect central and affine arrangements while generalizing both.
They have been studied by various authors including Aguiar-Mahajan \cite{AguiarMahajan}, Brown \cite{Brown}, Gente \cite{gente}, Zaslavsky \cite{zaslavsky77}, and, in Type A, the author together with Kim and Reiner \cite{DorpalenKimReiner}.
Moreover, cones are \emph{conditional oriented matroids}, as developed by Bandelt, Chepoi, and Knauer \cite{bck2018}.

Our extension of Gel'fand and Varchenko's work utilizes techniques inspired by the theory of Gr\"{o}bner bases.
As a consequence of this approach, we obtain a new result concerning the Varchenko-Gel'fand ring of an arrangement.
This result, Theorem \ref{thm:koszul}, is a Varchenko-Gel'fand analogue of a theorem of Peeva for the {\it Orlik-Solomon algebra} \cite{peeva}.

\subsection{Structure of this Paper}

In the remainder of this introduction, we state Theorems \ref{thm:main theorem} and \ref{thm:koszul} (Section \ref{sec:main results}) and then illustrate the main theorem with an example (Section \ref{sec:extended example}).
Some background on hyperplane arrangements, matroids, the Varchenko-Gel'fand ring, filtrations, and associated graded rings is given in Section \ref{sec:background}.
We prove Theorem \ref{thm:main theorem} in Section \ref{sec:definition of GGG} and we prove Theorem \ref{thm:koszul} in Section \ref{sec:supersolvability}.

\subsection{Statement of Main Results}\label{sec:main results}

Now let $\A = \{H_1,\dots,H_n\}$ be a central hyperplane arrangement in $V\cong\R^\ell$ and $\K$ a cone of $\A$.
An \emph{intersection} of $\A$ is an intersection $X = \bigcap_{i} H_i$ of some of the hyperplanes.
We use $\LL(\A)$ to denote the set of nonempty intersections of $\A$.

We say that $C\in\C(\A)$ is a \emph{chamber of the cone} $\K$ if it lies inside the open, convex set $\K$ and we say that $X\in \LL(\A)$ is an \emph{interior intersection} of $\K$ is an intersection $X\in \LL(\A)$ which cuts through the cone, i.e. for which $X\cap \K$ is nonempty. We denote the set of chambers of the cone by $\C(\K)$ and the set of interior intersections by $\LI(\K)$.
We are interested in the poset structure of $\LI(\K)$ ordered by reverse inclusion and denote its \emph{M\"{o}bius function} by $\mu$.
The poset $\LI(\K)$ is a ranked poset and the rank of each element $X\in\LI(\K)$ is its codimension, denoted $\text{codim}(X)$.

\medskip

The \emph{Varchenko-Gel'fand ring of the cone} $VG(\K)$ is the collection of maps $f: \C(\K)\to \Z$ with pointwise addition and multiplication.
Our main result takes its cue from a result of Varchenko and Gel'fand,
and interprets the \emph{Poincar\'{e} polynomial} of a cone
\[
\Poin(\K,t):= \sum_{X \in \LI(\K)}
|\mu(V,X)|~t^{\text{codim}(X)}
\]
as the \emph{Hilbert series} (defined below) of a certain \emph{graded} ring,
the associated graded ring for a certain filtration on $VG(\K)$.
The Poincar\'{e} polynomial is interesting in its own right. In \cite[Example A]{zaslavsky77}, for example, Zaslavsky showed that
the Poincar\'{e} polynomial evaluated at $1$ is precisely the number of chambers of $\A$ contained in the open cone $\K$, i.e.,
\(
\#\C(\K) = \Poin(\K,1)
\).
Further, several
interpretations for the Poincar\'{e} polynomial of $\K$ were recently given in the case where $\A$ is a Type A reflection arrangement \cite{DorpalenKimReiner}, but none was given for arbitrary hyperplane arrangements.

One can show that the Varchenko-Gel'fand ring $VG(\K)$ of a cone is generated (as a ring) by Heaviside functions associated to the hyperplanes of $\A$.
This endows $VG(\K)$ with a \emph{degree filtration} $\mathcal{F}=\{F_d\}$, where $F_d$ is the $\Z$-span of products of such Heaviside functions of degree at most $d$.
The \emph{associated graded ring} is
\[
\vgrr(\K)
:=\grr_{\mathcal{F}} VG(\K)
= \bigoplus_{d \geq 0} F_d/F_{d-1}
\]
and its Hilbert series is
\[
\Hilb(\vgrr(\K),t) := \sum_{d\geq 0} \text{rank}_\Z (F_d/F_{d-1})~t^d
\]
where $\text{rank}_\Z (F_d/F_{d-1})$ denotes the rank of $F_d/F_{d-1}$ as a $\Z$-module.
Varchenko and Gel'fand proved that when the cone $\K$ is the full arrangement, then $\vgrr(\A)$ is torsion-free and
$\Hilb(\vgrr(\A),t) = \Poin(\A,t)$ \cite{VarchenkoGelfand}. We will show the same equality for cones and, specializing to the full arrangement, obtain a novel proof of Varchenko and Gel'fand's original result.

Our proof comes from giving a generating set $\GGG$ of relations that plays the role of a Gr\"{o}bner basis presentation for $VG(\K)$ and $\vgrr(\K)$ as quotients of $\Z[e_1,\dots, e_n]$; when working over
a field instead of $\Z$, they are Gr\"obner bases.
The relations in $\GGG\subseteq \Z[e_1,\dots, e_n]$ are summarized in Figure \ref{chart}, where we have made the (harmless) assumption that $\K$ is an intersection of \emph{(open) positive halfspaces}, i.e.
\(
\K = \bigcap_{i\in W} H_i^+
\)
where $W\subseteq [n]$ and $H_i^+  := \{\mathbf{x}\in\R^\ell \mid v_i\cdot \mathbf{x} >0\}$ for some choice of normal vector $v_i$ to $H_i$.
For a subset $S$ of $[n]$, we use the notation $e_S:= \prod_{i\in S} e_i$ to describe a squarefree monomial indexed by $S$ in the variables $e_1,\dots, e_n$.
Our main theorem will assert that the elements of $\GGG$ (given in the second column of Figure \ref{chart}) and $\{\init_{\deg}(g)\mid g\in \GGG\}$ (given in the third column) give presentations for $VG(\K)$ and $\vgrr(\K)$, respectively.

The most interesting polynomials in $\GGG$ are defined by \emph{signed circuits}. For each hyperplane $H_i$ of $\A$, fix a normal vector $v_i$.
A pair of disjoint subsets  $D^+,D^-\subseteq [n]$ is a \emph{signed dependency} if there is a linear relation
\(
\sum_{i} \lambda_i v_i = 0
\)
where $i\in D^+$ if $\lambda_i <0$ and $i\in D^-$ if $\lambda_i >0$.
To keep track of signed dependencies, we write them as tuples $D = (D^+,D^-)$ and denote the \emph{underlying (unsigned) set} by $\underline{D} = D^+ \cup D^-$. Our presentations for $VG(\K)$ and $\vgrr(\K)$ concern signed circuits $C = (C^+,C^-)$, which are signed dependencies for which $\underline{C}$ is minimal under inclusion.

In the main theorem, our presentations of $VG(\K)$ and $\vgrr(\K)$ will give a $\Z$-basis for both rings in terms of a certain family of monomials indexed by \emph{$\K$-no broken circuit sets}, which we describe now.
One can \emph{break} a circuit $C = (C^+,C^-)$, by removing the smallest-indexed element $i_0$ of $\underline{C}$. The resulting (unsigned) set  $\underline{C}\backslash\{i_0\}$ is called a \emph{broken circuit}.
A subset of $[n]$ not containing any broken circuits is called a \emph{no broken circuit set}. A no broken circuit set $N$ is, furthermore, a \emph{$\K$-no broken circuit set} (hereafter $\K$-NBC set) if
\(
X=\bigcap_{i\in N} H_i
\)
cuts through the cone $\K$, meaning $X \in \LI(\K)$. We denote the set of $\K$-NBC sets by $NBC(\K)$.

$\K$-NBC sets naturally arise when one studies the Poincar\'{e} polynomial of the cone.
Since every lower interval $[V,X]$ of $\LI(\K)$ is isomorphic to the corresponding lower interval $[V,X]$ in $\LL(\A)$, a theorem of Rota \cite[Section 7]{rota} (c.f. \cite[Theorem 1.1]{sagan})
allows us to compute the M\"{o}bius function of the interval $[V,X]$ in $\LI(\K)$ via the $\K$-NBC sets, i.e., for all $X\in \LI(\K)$
 \[
 \mu(V, X) = (-1)^{\text{codim}(X)} \cdot \#\left\{N\in NBC(\A)~\Big\vert~ \bigcap_{i\in N}H_i = X\right\}.
 \]
As a result, the Poincar\'{e} polynomial of a cone also has an expression in terms of $\K$-NBC sets:
\[
\Poin(\K,t)= \sum_{N \in NBC(\K)}
t^{\# N}
\]
Taking $t=1$, gives
\begin{equation}
\label{cone-nbc-sets-equinumerous-with-cone-chambers}
\#\C(\K)=\#NBC(\K).
\end{equation}

\begin{thm}
\label{thm:main theorem}
Let $\K$ be a cone of an arrangement $\A$ and $\GGG$ the relations from Figure \ref{chart}.
Choose a monomial order $\prec$ on $\Z[e_1,\ldots,e_n]$ which refines the ordering by degree.
Then $VG(\K), \vgrr(\K)$ have
presentations
\[
\begin{aligned}
VG(\K) &\cong \Z[e_1,\ldots,e_n]/(\GGG),\\
 \vgrr(\K) &\cong \Z[e_1,\ldots,e_n]/(\init_{\deg}(\GGG)),
\end{aligned}
\]
and free $\Z$-modules bases given by the images of the
$\K$-NBC monomials $\{e_N\}_{N \in NBC(\K)}$.
In particular,
$$
\Hilb(\vgrr(\K),t)=\sum_{N \in NBC(\K)} t^{\#N} =\Poin(\K,t).
$$
\end{thm}

\noindent \changes{Recall that $\K = \bigcap_{i\in W} H_i^+$ for some (possibly-redundant) $W\subseteq[n]$.
The following corollary is an immediate consequence of Theorem \ref{thm:main theorem}.}

\begin{coro}
$VG(\K) \cong \Z[e_1,\dots e_n]/I_\K$ where $I_\K$ is generated by
  \begin{itemize}
  \item {\sf (Idempotent)} $e_{i}^2 - e_i$ for $i\in[n]$,
  \item {\sf(Unit)} $e_i-1$ for $i\in W$,
  \item {\sf (Circuit)} $\displaystyle (e_{C^+})  \prod_{j\in C^-} (e_j-1) - (e_{C^-})\prod_{i\in C^+} (e_i-1)$ for signed circuits $C = (C^+, C^-)$.
\end{itemize}
\end{coro}

Theorem~\ref{thm:main theorem} has an interesting consequence when $\A$ is a \emph{supersolvable arrangement}: we obtain a Varchenko-Gel'fand ring analogue of a result proven by Peeva for the \emph{Orlik-Solomon algebra} of $\A$. In order to state this result, we work with coefficients in a field $\kk$,
and consider the ring $VG_{\kk}(\A)$ of maps from the chambers of $\A$ to a field $\kk$, denoting the associated graded ring by $\vgrr_{\kk}(\A)$.

\begin{thm}\label{thm:koszul}
  If $\A$ is a supersolvable arrangement, then $\vgrr_{\kk}(\A)$ is Koszul.
\end{thm}

This result uses Bj\"{o}rner-Ziegler's characterization of supersolvable arrangements via broken circuit sets \cite{bjorner-ziegler}, so one might hope that it extends to cones of supersolvable arrangements.
Unfortunately, there are cones of supersolvable arrangements for which $\vgrr_{\kk}(\K)$ is not Koszul (we provide such an example in Section~\ref{sec:supersolvability}).

\begin{figure}
\resizebox{\linewidth}{!}{
$$
\begin{tabular}{|c|c|c|c|c|}\hline
& $g \in \GGG$ & $\init_{\deg}(g)$ & $\init_\prec(g)$ \\ \hline\hline
& \text{For all }$i \in [n],$ & & \\
{\sf Idempotent}  & & & \\
& $e_i^2 -e_i$&  $e_i^2$ &  $e_i^2$ \\
&  & & \\ \hline
& \text{For all }$i \in W,$ & & \\
& & & \\
{\sf Unit} & $e_i-1$&  $e_i$ &  $e_i$ \\
 & & & \\ \hline
& $\text{For all circuits }C=(C^+,C^-)$ & & \\
& $\text{ with }\emptyset \neq W \cap C^+=W \cap \underline{C},$ & & \\
{\sf Cone}& & & \\
{\sf Circuit} & $e_{C^+\setminus W} \cdot \displaystyle\prod_{j \in C^-} (e_j-1)$&  $e_{\underline{C} \setminus W}$ &  $e_{\underline{C} \setminus W}$  \\
 & & & \\
& $(\text{similarly if }\emptyset \neq W \cap C^-=W \cap \underline{C})$ & & \\ \hline
& $\text{For signed circuits }C=(C^+, C^-)$ & & \\
& $\text{ with }\emptyset = W \cap \underline{C},$ & & \\
 & & & \\
{\sf Circuit} & $(e_{C^+}) \displaystyle\prod_{j \in C^-} (e_j-1)
-(e_{C^-}) \displaystyle \prod_{j \in C^+} (e_j-1)$
&  $+ \displaystyle \sum_{i \in C^+} e_{\underline{C} \setminus \{i\}}
- \displaystyle \sum_{j \in C^-} e_{\underline{C} \setminus \{j\}}$
&  $e_{\underline{C} \setminus \{i_0\}}$  \\
& & &  $\text{ where }i_0:=\min_\prec(\underline{C})$\\
&  & & \\ \hline
\end{tabular}
$$
}
\caption{The relations $g$ in $\GGG$, along with their degree-initial form
$\init_{\deg}(g)$ and their initial term $\init_{\prec}(g)$ for
any monomial order $\prec$ that satisfies $e_1\prec\cdots\prec e_n$.}\label{chart}
\end{figure}

\subsection{An extended example}\label{sec:extended example}
In this section, we give an extended example illustrating Theorem \ref{thm:main theorem}.
Consider the cone $\K$ of a central arrangement in $\R^3$ of which an affine slice is drawn below on the left.
We can compute the Poincar\'{e} polynomial of the cone from $\LI(\K)$ (below, on the right).
It is $\Poin(\K,t) = 1+3t+t^2$.
\begin{center}
\begin {tikzpicture}[scale=.5]

\node[invisivertex] (A1) at (-1,1){$H_4$};
\node[invisivertex] (A2) at (3,-3){};
\node[invisivertex] (B1) at (.5,-3){};
\node[invisivertex] (B2) at (.75,3){$H_1$};
\node[invisivertex] (C1) at (2.3,-3.4){$H_2$};
\node[invisivertex] (C2) at (-.25,2.75){};
\node[invisivertex] (D1) at (0,-3){};
\node[invisivertex] (D2) at (2.5,3.4){$H_3$};
\node[invisivertex] (E1) at (-1,-1){$H_5$};
\node[invisivertex] (E2) at (3,3){};

\path[-] (A1) edge [bend left =0,color=teal, dashed] node[above] {} (A2);
\path[-] (E1) edge [bend left =0,color=teal, dashed] node[above] {} (E2);
\path[-] (B1) edge [bend left =0] node[above] {} (B2);
\path[-] (C1) edge [bend left =0] node[above] {} (C2);
\path[-] (D1) edge [bend left =0] node[above] {} (D2);

\fill[teal!40,nearly transparent] (0,0) -- (3,3) -- (3,-3) -- cycle;
\end{tikzpicture}
\hspace{2cm}
\begin {tikzpicture}[scale=.5]

\node[invisivertex] (0) at (3,0){$\hat{0}$};
\node[invisivertex] (H1) at (0,2){$H_{1}$};
\node[invisivertex] (H2) at (3,2){$H_{2}$};
\node[invisivertex] (H3) at (6,2){$H_{3}$};
\node[invisivertex] (X23) at (4.5,4){$H_{2}\cap H_3$};

\path[-] (0) edge [bend left =0] node[above] {} (H2);
\path[-] (0) edge [bend left =0] node[above] {} (H3);
\path[-] (0) edge [bend left =0] node[above] {} (H1);

\path[-] (H2) edge [bend left =0] node[above] {} (X23);
\path[-] (H3) edge [bend left =0] node[above] {} (X23);

\end{tikzpicture}
\end{center}

\begin{figure}
\begin{tabular}{|c|c|c|c|}
\hline
& $g\in\GGG$ & $\init_{\deg}(g)$ & $\init_{\prec}(g)$ \\
\hline
\hline
&&&\\
& For all $i\in [5]$ & &\\
{\sf (Idempotent)} & $e_i^2 - e_i$ & $e_i^2$ & $e_i^2$  \\
&&&\\
\hline
&&&\\
{\sf (Unit)} & $e_4 - 1,$ & $e_4$, & $e_4$,\\
& $e_5 - 1$ & $e_5$ &  $e_5$ \\
&&&\\
\hline
&&&\\
{\sf (Circuit)} & (none) & (none) & (none) \\
&&&\\
\hline
&&&\\
{\sf (Cone-Circuit)} & $e_1e_2 - e_2$, & $e_1e_2$, & $e_1e_2$,\\
& $e_1e_2e_3 - e_2e_3$, & $e_1e_2e_3$, & $e_1e_2e_3$,\\
& $e_1e_3 - e_3$ & $e_1e_3$ & $e_1e_3$\\
&&&\\
\hline
\end{tabular}
\caption{The elements of $\GGG$ for some choice of orientation in the extended example in Section \ref{sec:extended example}.
We omit the the redundant {\sf Cone-Circuit} relations for which the opposite orientation is already given. }
\label{fig:extended-example}
\end{figure}

\noindent Figure \ref{fig:extended-example} shows the relations $\GGG$ for some choice of orientation of $\A$ (we omit the the redundant {\sf Cone-Circuit} relations for which the opposite orientation is already given).
Our main theorem says that
$VG(\K)\cong \Z[e_1,e_2,e_3,e_4,e_5]/(\GGG)$
and the associated graded ring has presentation
$\vgrr(\K)\cong \Z[e_1,e_2,e_3,e_4,e_5]/(\init_{\deg}\GGG)$.
Furthermore
\[
\vgrr(\K) \cong \Z\cdot\{1\}\oplus \Z\cdot\{x_1,x_2,x_3\} \oplus \Z\cdot\{x_2x_3\}\,,
\]
which means that $\Hilb(\vgrr(\K),t) = 1+3t+t^2$.

\section{Background}\label{sec:background}

\subsection{Hyperplanes, Cones and Linear Algebra}

Let $\A = \{H_1,\dots, H_n\}$ be a central hyperplane arrangement with $n$ distinct hyperplanes $H_i = \{\mathbf{x}\in\R^\ell \mid v_i \cdot \mathbf{x} = 0\}$ for normal vectors $\{v_i\}_{i\in[n]}$. A \emph{chamber} of $\A$ is an open, connected component of
$V\backslash \bigcup_{H\in\A}H$ and an
\emph{intersection} $X$ of $\A$ is the subspace of $V$ defined by intersecting some of the hyperplanes of $\A$. We denote the set of chambers and set of intersections of $\A$ by $\C(\A)$ and $\LL(\A)$, respectively. Note that $\LL(\A)$ always contains $V$, the intersection of none of the hyperplanes.
When ordered by reverse inclusion, $\LL(\A)$ is ranked by codimension and satisfies these two conditions that define a \emph{geometric lattice} \cite[Definition 3.9]{stanley-hyperplanes}:
\begin{itemize}
    \item (Upper Semi-Modular) For all $X,Y\in \LL(\A)$, the codimensions of $X$ and $Y$ satisfy
    \[
    \text{codim}(X) + \text{codim}(Y) \geq \text{codim}(X\vee Y) +  \text{codim}(X \wedge Y)
    \]
    where $X \vee Y$ is the intersection $X \cap Y$, and $X\wedge Y$ denotes the lowest-dimensional subspace $Z\in \LL(\A)$ containing both $X$ and $Y$.
    \item (Atomic) Every $X\in \LL(\A)$ is an intersection of some of the hyperplanes of $\A$.
\end{itemize}

Our choice of normal vectors $v_1,\dots,v_n$ induces an orientation of $\A$, which we use to define a pair of halfspaces $H^+$ and $H^-$ for every hyperplane $H\in\A$:
\begin{align*}
    H^+ & = \{\mathbf{x}\in\R\mid v_i\cdot \mathbf{x} >0\}\\
    H^- & = \{\mathbf{x}\in\R\mid v_i\cdot \mathbf{x}<0\}.
\end{align*}
With this notation, each hyperplane yields a decomposition of $V$ into three sets: the positive halfspace $H^+$, the negative halfspace $H^-$, and the hyperplane itself. To emphasize this relationship, we sometimes use a superscript $0$ to denote the hyperplane itself, thus decomposing $V$ into $H^+ \sqcup H^- \sqcup H^0$.

A \emph{cone} $\K$ of an arrangement $\A$ is an intersection of halfspaces defined by some of the hyperplanes of $\A$. To simplify notation, we assume that $\A$ is oriented so that $\K$ is an intersection of positive halfspaces.
\changes{We use $W \subseteq [n]$ to denote the set of indices of a (potentially redundant) set of defining hyperplanes of $\K$, chosen among the hyperplanes of $\A$.
We call $W$ the set of \emph{walls} of the cone.
Note that there may be several choices of walls which define the same cone $\K$ (when $\K$ is viewed as a subset of the vector space).
Our main theorem will show that these different choices of walls for $\K$ produce the same Varchenko-Gelfand ring.}
%

The notions of chambers and intersections naturally extend to cones: a chamber of a cone $\K$ is a chamber of $\A$ contained in the open set $\K$
and an intersection $X$ of $\K$ is an intersection of $\A$ with $X\cap\K\not=\emptyset$. We denote the set of chambers and intersections of $\K$ by $\C(\K)$ and $\LI(\K)$, respectively.
Although $\LI(\K)$ is \emph{not} a geometric lattice when ordered by reverse inclusion, every lower interval, $[V,X]$ for $X\in\LI(\K)$, \emph{is} a geometric lattice.

In the introduction, we defined the  Poincar\'{e} polynomial of a cone and showed that it can be computed from the interior intersection poset by
\[
{\sf Poin}(\K,t) =
\sum_{X\in\LI(\K)} |\mu(V,X)|~t^{rk(X)}.
\]
We noted, furthermore, that a theorem of Zaslavsky relates ${\sf Poin}(\K,t)$ to the number of chambers in a cone. Combining Zaslavsky's theorem with the M\"{o}bius function definition of the Poincar\'{e} polynomial (above) gives
\[
\#\C(\K) = \sum_{X\in \LI(\K)} |\mu(V,X)|.
\]
In the introduction, we gave another interpretation of the Poincar\'{e} polynomial in terms of the $\K$-NBC sets, which arise from thinking about the oriented matroid of the arrangement. We will introduce these sets in Section \ref{sec:background:oriented matroids} and here give an example of the poset-interpretation of $\Poin(\K,t)$.

\begin{example}\label{example:cones}
Consider the arrangement in $\R^2$ below on the left with the given orientation. The cone $\K = H_1^+$ with $W=\{1\}$ is the shaded region below on the right.

    
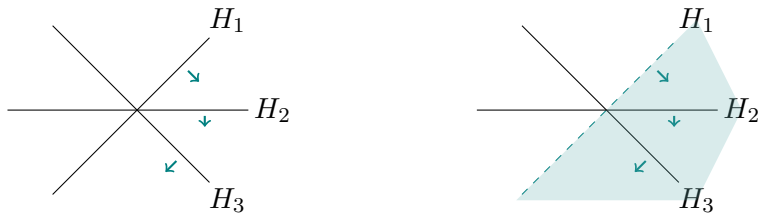
\begin{center}
      \begin {tikzpicture}[scale=.6]

      \node[invisivertex] (A1) at (-2,-2){};
      \node[invisivertex] (A2) at (2,2){$H_1$};
      \node[invisivertex] (B1) at (-2,2){};
      \node[invisivertex] (B2) at (2,-2){$H_3$};
      \node[invisivertex] (C1) at (-3,0){};
      \node[invisivertex] (C2) at (3,0){$H_2$};

    \node[invisivertex] (OA1) at (1,1){};
      \node[invisivertex] (OA2) at (1.5,.5){};
      \node[invisivertex] (OB1) at (1,-1){};
      \node[invisivertex] (OB2) at (.5,-1.5){};
      \node[invisivertex] (OC1) at (1.5,0){};
      \node[invisivertex] (OC2) at (1.5,-.5){};

      \path[-] (A1) edge [bend left =0] node[above] {} (A2);
      \path[-] (B1) edge [bend left =0] node[above] {} (B2);
      \path[-] (C1) edge [bend left =0] node[above] {} (C2);

  \path[->] (OA1) edge [color=teal, thick] node[above] {} (OA2);
      \path[->] (OB1) edge[color=teal, thick] node[above] {} (OB2);
      \path[->] (OC1) edge [color=teal, thick] node[above] {} (OC2);

      \end{tikzpicture}
\hspace{2cm}
    \begin {tikzpicture}[scale=.6]

      \node[invisivertex] (A1) at (-2,-2){};
      \node[invisivertex] (A2) at (2,2){$H_1$};
      \node[invisivertex] (B1) at (-2,2){};
      \node[invisivertex] (B2) at (2,-2){$H_3$};
      \node[invisivertex] (C1) at (-3,0){};
      \node[invisivertex] (C2) at (3,0){$H_2$};

      \node[invisivertex] (OA1) at (1,1){};
      \node[invisivertex] (OA2) at (1.5,.5){};
      \node[invisivertex] (OB1) at (1,-1){};
      \node[invisivertex] (OB2) at (.5,-1.5){};
      \node[invisivertex] (OC1) at (1.5,0){};
      \node[invisivertex] (OC2) at (1.5,-.5){};

      \path[-] (A1) edge [bend left =0,dashed, teal] node[above] {} (A2);
      \path[-] (B1) edge [bend left =0] node[above] {} (B2);
      \path[-] (C1) edge [bend left =0] node[above] {} (C2);

     \path[->] (OA1) edge [color=teal, thick] node[above] {} (OA2);
      \path[->] (OB1) edge[color=teal, thick] node[above] {} (OB2);
      \path[->] (OC1) edge [color=teal, thick] node[above] {} (OC2);

      \fill[teal!60,nearly transparent] (-2,-2) -- (2,2) -- (3,0) -- (2,-2) -- cycle;
      \end{tikzpicture}
\end{center}

%
%
%
%
%
%
%
%
The Hasse diagrams of $\LL(\A)$ (left) and $\LI(\K)$ (right) are below, together with $\mu(V, X)$ for each $X$ (circled value).
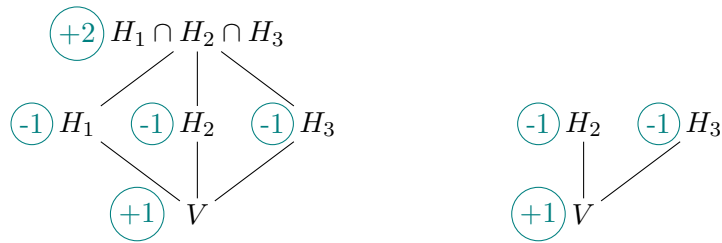
\begin{center}
\begin {tikzpicture}[scale=.8]

\node[invisivertex] (0) at (3,0){$V$};
\node[moebius] (M0) at (2,0){+1};
\node[invisivertex] (H1) at (1,1.5){$H_{1}$};
\node[moebius] (M1) at (0.25,1.5){-1};
\node[invisivertex] (H2) at (3,1.5){$H_{2}$};
\node[moebius] (M2) at (2.25,1.5){-1};
\node[invisivertex] (H3) at (5,1.5){$H_{3}$};
\node[moebius] (M2) at (4.25,1.5){-1};
\node[invisivertex] (X23) at (3,3){$H_1\cap H_{2}\cap H_3$};
\node[moebius] (M123) at (1,3){+2};

\path[-] (0) edge [bend left =0] node[above] {} (H2);
\path[-] (0) edge [bend left =0] node[above] {} (H3);
\path[-] (0) edge [bend left =0] node[above] {} (H1);

\path[-] (H1) edge [bend left =0] node[above] {} (X23);
\path[-] (H2) edge [bend left =0] node[above] {} (X23);
\path[-] (H3) edge [bend left =0] node[above] {} (X23);
\end{tikzpicture}
\hspace{2cm}
\begin {tikzpicture}[scale=.8]

\node[invisivertex] (0) at (3,0){$V$};
\node[moebius] (M0) at (2.25,0){+1};
\node[invisivertex] (H2) at (3,1.5){$H_{2}$};
\node[moebius] (M2) at (2.25,1.5){-1};
\node[invisivertex] (H3) at (5,1.5){$H_{3}$};
\node[moebius] (M2) at (4.25,1.5){-1};

\path[-] (0) edge [bend left =0] node[above] {} (H2);
\path[-] (0) edge [bend left =0] node[above] {} (H3);

\end{tikzpicture}
\end{center}
The associated Poincar\'{e} polynomials are ${\sf Poin}(\A,t) = 1+3t+2t^2$ and ${\sf Poin}(\K,t) = 1+2t$, respectively. It is easy to see from the pictures in Example \ref{example:cones} that the arrangement has $1+3+2 = 6$ chambers and that the cone has $1+2=3$ chambers.
\end{example}

\subsection{Oriented Matroids}\label{sec:background:oriented matroids}

The collection of normal vectors $\{v_1,\dots,v_n\}$ of $\A$ naturally gives rise to an \emph{oriented matroid}.
The theory of (oriented) matroids
arising from hyperplane arrangements is well-studied and there are many excellent sources on this topic including \cite{BjornerEtAl}, \cite[Section 2.1]{OrlikTerao}, and \cite[Lecture 3]{stanley-hyperplanes}. We briefly review some basics but refer the reader to the preceding sources for a more detailed discussion.

Let $E$ be a finite set. A \emph{signed set} $D = (D^+,D^-)$ of $E$ is a disjoint, ordered pair of subsets $D^+,D^-\subseteq E$.
For a signed subset $D$ of $E$ and $e\in E$ and $D$, define
\[
D_e := \begin{cases}
+ & \text{if }e\in D^+\\
- & \text{if }e\in D^-\\
0 & \text{else}.
\end{cases}
\]
The collection $\underline{D}:=\{e\in E: D_e \not=0\}$ is called the \emph{(unsigned) support set} of $E$.  Each signed set $D$ has an ``opposite" signed set $-D$ with the same (unsigned) support set but
\(
(-D)_e = -(D_e)
\)
for all $e\in\underline{D}$. For an arbitrary pair $C, D$ of signed sets, we use their \emph{separating set} to keep track of the places where their signs are opposite, i.e. if $C,D$ are signed sets of $E$ then the separating set of $C$ and $D$ is
\[
S(C,D) := \{e\in E\mid C_e = - D_e \not=0\}.
\]
Its easy to see that the separating set of $D$ and $-D$ is $\underline{D}$.

Finally we define the \emph{composition product} of two signed sets $C,D$ with the same ground set $E$. The composition of $C$ with $D$ is the signed set $C\circ D$ where
\[
(C\circ D)_e =
\begin{cases}
C_e &\text{if } C_e \not=0 \\
D_e &\text{else}
\end{cases}
\qquad \text{for }e\in E.
\]

\begin{defn}
Let $E$ be a finite set and $\DDD$ a collection of signed subsets of $E$. The pair $\DDD$ is the set of vectors of an \emph{oriented matroid} on $E$
if $\DDD$ satisfies the \emph{vector axioms}:
\begin{itemize}
  \item[{\sf V0}.] $\mathbb{0}\in\DDD$
  \item[{\sf V1}.] If $D\in\DDD$, then $-D\in\DDD$.
  \item[{\sf V2}.] If $C,D\in\DDD$, then $C\circ D\in\DDD$
  \item[{\sf V3}.] If $C,D\in\DDD$ and $e\in S(C,D)$ then there exists $F\in\DDD$ with
  \begin{align*}
    F_e & = 0\\
    F_f & = (C\circ D)_f & \text{for }f\in E\backslash S(C,D).
  \end{align*}
\end{itemize}
\end{defn}

The main proof of this paper concerns the connection between an oriented matroid and its \emph{dual}. For an oriented matroid $M = (E,\mathfrak{D})$, there is a unique oriented matroid $M^* = (E,\mathfrak{D}^*)$ with vectors
\[
\mathfrak{D}^*
=
\{
(F^+,F^-) \mid D\perp F \text{ for all }D\in \mathfrak{D}
\}
\]
where $F\perp D$ if $F$ and $D$ are \emph{orthogonal}, i.e. $\{F_e \cdot D_e \mid e\in E\}$ either equals $\{0\}$ or contains $\{+,-\}$ \cite[\S 6.2.5]{richter-gebert-ziegler}.
We call $M^*$ the dual matroid to $M$ and call $\mathfrak{D}^*$ the \emph{covectors} of $M$.
One can show that the covectors of $M$ also satisfy the vector axioms so that the dual oriented matroid is in fact an oriented matroid.

We will be concerned with oriented matroids defined by sets of vectors $ \{v_1,\ldots,v_n\}$ in $\R^d$, which naturally come equipped with \emph{signed dependencies} given by linear combinations. Whenever $\sum_{v\in\underline{D}} \lambda_v v = \mathbf{0}$, one has a signed dependency $D = (D^+,D^-)$ where
\[
D_v =
\begin{cases}
+ & \text{if }\lambda_v>0\\
- & \text{if }\lambda_v<0\\
0 & \text{else}.
\end{cases}
\]
In this context, one can think of the composition product of as a sum of dependencies where the second dependency is multiplied by a small, positive number.

The covectors of this oriented matroid also have a well-known geometric interpretation via (nonempty) intersections of halfspaces defined by some of the hyperplanes of $\A$, see \cite[Section 1.1.3]{AguiarMahajan} for example. We will use the fact that if an intersection $\bigcap_{i} H_i^{\varepsilon_i}$ is nonempty for $\{\varepsilon_i\}_i$, then there is a covector $F\in\mathfrak{D}^*$ such that $F_i = \varepsilon_i$ for all $i\in\underline{F}$.

The minimal, nonempty signed dependencies of an oriented matroid are called \emph{signed circuits} and we denote the set of all signed circuits of $\DDD$ by $\CCC$, i.e.
\[
\CCC: = \{C\in \DDD  \mid \underline{C}
\text{ is minimal under inclusion, }C\not=\mathbb{0} \}.
\]
A theorem of Bland-Las Vergnas and Edmonds-Mandel \cite[Theorem 3.7.5]{BjornerEtAl} says that every vector $D\in\DDD$ is a composition of signed circuits, i.e. there is some $k\in\Z$ and collection $C^{(1)},\dots, C^{(k)}\in\CCC$ such that
\[
D = C^{(1)} \circ C^{(2)} \circ \cdots \circ C^{(k)}.
\]
Furthermore, one can select the circuits of this composition so that they \emph{conform} to $D$, meaning that for all $i=1,\dots,k$ and $e\in \underline{D}$: if $C^{(k)}_e$ is nonzero then $C^{(i)}_e = D_e$ \cite[Proposition 3.7.2]{BjornerEtAl}. We will use a weaker version in the proof of our main theorem: every signed dependence $D$ can be written as a composition of circuits $D = C^{(1)} \circ C^{(2)} \circ \cdots \circ C^{(k)}$ and it is easy to see that the first circuit $C^{(1)}$ always conforms to $D$.

In order to simplify notation, we will hereafter conflate a vector $v_i$ with its index $i$. We take $E = [n]$, so that $\CCC$ is a collection of signed subsets of $[n]$.
For $C\in\CCC$, we say that $\underline{C}-\{i\}$ is a \emph{broken circuit} if $i$ is the smallest index (under the usual order on $[n]$ in which $1<2<3<\cdots < n$) such that $i\in C$. We will also consider the \emph{no broken circuit sets} of $\A$, denoted $NBC(\A)$, which are the subsets of $[n]$ containing no broken circuits.
 A no broken circuit set $N\in NBC(\A)$ of $\A$ is a \emph{$\K$ no broken circuit set} or $\K$-NBC set if
\[
\bigcap_{i\in N} H_i^0 \in \LI(\K).
\]
If we take our cone to be the intersection of no halfspaces, i.e. \changes{the set of walls is empty,} then we recover the full arrangement, and the $\K$-NBC sets are precisely the usual NBC sets of the arrangement. We will denote the set of $\K$-NBC sets by $NBC(\K)$.
In the introduction, we saw that the $\K$-NBC sets provide a secondary description for the Poincar\'{e} polynomial
\[
\Poin(\K,t)= \sum_{N \in NBC(\K)}
t^{\# N}.
\]
By setting $t=1$, we obtain Equation \eqref{cone-nbc-sets-equinumerous-with-cone-chambers}, which says
\(
\#\C(\K) = \#NBC(\K).
\)
This equality will be central in our understanding of the Varchenko-Gel'fand ring, which we turn to now.

\subsection{The Varchenko-Gel'fand Ring}
The \emph{Varchenko-Gel'fand ring} of an arrangement $\A$ is the ring of maps $f:\C(\A)\to \Z$ under pointwise addition and multiplication \cite{VarchenkoGelfand}. Similarly, we define the \emph{Varchenko-Gel'fand ring of a cone} $\K$ to be the ring with underlying set
\[
VG(\K) = \{f:\C(\K)\to \Z\}
\]
under pointwise addition and multiplication. We can represent elements of $VG(\K)$ as a labelling of the chambers of $\K$ with integers.

\begin{example}
Consider the cone from Example \ref{example:cones}. Below are several elements of $VG(\K)$:

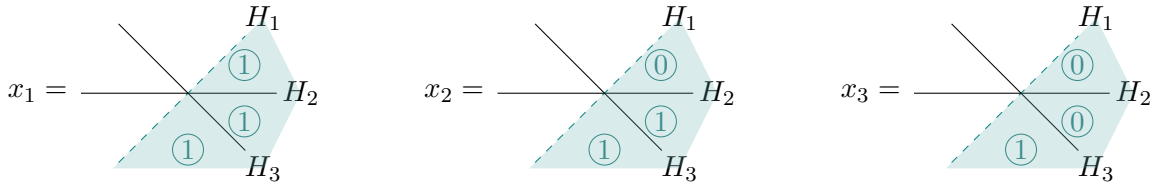
\bigskip

    \begin{center}
    \begin {tikzpicture}[scale=.5]

        \node[invisivertex] (f) at (-4,0){$x_1 = $};

      \node[invisivertex] (A1) at (-2,-2){};
      \node[invisivertex] (A2) at (2,2){$H_1$};
      \node[invisivertex] (B1) at (-2,2){};
      \node[invisivertex] (B2) at (2,-2){$H_3$};
      \node[invisivertex] (C1) at (-3,0){};
      \node[invisivertex] (C2) at (3,0){$H_2$};

      \node[chamberlabel] (L1) at (1.5,.75){$1$};
      \node[chamberlabel](L2) at (1.5,-.75){$1$};
      \node[chamberlabel] (L3) at (0,-1.5){$1$};

      \path[-] (A1) edge [bend left =0, dashed,teal] node[above] {} (A2);
      \path[-] (B1) edge [bend left =0] node[above] {} (B2);
      \path[-] (C1) edge [bend left =0] node[above] {} (C2);

      \fill[teal!60,nearly transparent] (-2,-2) -- (2,2) -- (3,0) -- (2,-2) -- cycle;
      \end{tikzpicture}
\hspace{1cm}
    \begin {tikzpicture}[scale=.5]

    \node[invisivertex] (f) at (-4,0){$x_2 = $};

      \node[invisivertex] (A1) at (-2,-2){};
      \node[invisivertex] (A2) at (2,2){$H_1$};
      \node[invisivertex] (B1) at (-2,2){};
      \node[invisivertex] (B2) at (2,-2){$H_3$};
      \node[invisivertex] (C1) at (-3,0){};
      \node[invisivertex] (C2) at (3,0){$H_2$};

   \node[chamberlabel] (L1) at (1.5,.75){$0$};
    \node[chamberlabel] (L2) at (1.5,-.75){$1$};
    \node[chamberlabel] (L3) at (0,-1.5){$1$};

      \path[-] (A1) edge [bend left =0,dashed,teal] node[above] {} (A2);
      \path[-] (B1) edge [bend left =0] node[above] {} (B2);
      \path[-] (C1) edge [bend left =0] node[above] {} (C2);

      \fill[teal!60,nearly transparent] (-2,-2) -- (2,2) -- (3,0) -- (2,-2) -- cycle;
     \end{tikzpicture}
     \hspace{1cm}
    \begin {tikzpicture}[scale=.5]

    \node[invisivertex] (f) at (-4,0){$x_3 = $};

      \node[invisivertex] (A1) at (-2,-2){};
      \node[invisivertex] (A2) at (2,2){$H_1$};
      \node[invisivertex] (B1) at (-2,2){};
      \node[invisivertex] (B2) at (2,-2){$H_3$};
      \node[invisivertex] (C1) at (-3,0){};
      \node[invisivertex] (C2) at (3,0){$H_2$};

    \node[chamberlabel] (L1) at (1.5,.75){$0$};
    \node[chamberlabel] (L2) at (1.5,-.75){$0$};
    \node[chamberlabel] (L3) at (0,-1.5){$1$};

      \path[-] (A1) edge [bend left =0,dashed,teal] node[above] {} (A2);
      \path[-] (B1) edge [bend left =0] node[above] {} (B2);
      \path[-] (C1) edge [bend left =0] node[above] {} (C2);

      \fill[teal!60,nearly transparent] (-2,-2) -- (2,2) -- (3,0) -- (2,-2) -- cycle;
      \end{tikzpicture}
\end{center}
\end{example}

In the preceding example, the elements are suggestively labelled $x_1, x_2,$ and $x_3$ to represent \emph{Heaviside functions} (defined below) given by some orientation of the hyperplanes $H_1,H_2,$ and $H_3$. The Heaviside function associated to $H_3$ is not included, as it would be $1$ on every chamber of the cone.

\medskip

In Varchenko and Gel'fand's original paper \cite{VarchenkoGelfand}, they observe that $VG(\A)$ is generated as a $\Z$-algebra by Heaviside functions
\[
x_i(C)
=
\begin{cases}
1 & \text{if}~C\subseteq H_i^+ \\
0 & \text{else}
\end{cases}
\qquad~\qquad \text{for }C\in\C(\A).
\]
for each hyperplane $H_i\in \LL(\A)$.
It suffices to check that if $f:\C(\A) \to \Z$
is
\[
f = \sum_{C\in \C(\A)} f(C) ~ \prod_{\substack{i\in W_C\\ C\subseteq H_i^+}} x_i
\prod_{\substack{j\in W_C\\ C\subseteq H_j^-}} (1 - x_i).
\]
where $W_C := \{i \mid H_i \cap \overline{C} \not= \emptyset\}\subseteq[n]$ is the set of hyperplanes which have a nonempty intersection the closure of $C$.
Their proof extends without modification to the cone case, where $x_i(C)$ is $1$ when both $C \subseteq H_i^+$ and $C\in\C(\K)$, and $0$ otherwise.
When $W$ is a choice of walls for the cone $\K$, this means that $x_i \equiv 1$ for each $i\in W$.

\begin{remark}
We will usually view the Varchenko-Gel'fand ring of a cone as a ring of functions $\C(\K)\to \Z$.
However, it is also a quotient of $VG(\A)$:
one has a surjective \emph{restriction map}
\(
{\sf res}: VG(\A) \twoheadrightarrow VG(\K)
\)
sending $f \mapsto f|_{\K}$
defined by
\(
f|_{\K}(C) := f(C)
\)
for $C\in\C(\K)$.
\end{remark}

In the previous section, we introduced oriented matroids and defined a family of sets called the $\K$-NBC sets. By the definition of the Varchenko-Gel'fand ring, we have $VG(\K) \cong \Z^{\#\C(\K)}$. Combining this isomorphism with Equation \eqref{cone-nbc-sets-equinumerous-with-cone-chambers} implies
\[
VG(\K) \cong \Z^{\#\C(\K)} = \Z^{\#NBC(\K)}.
\]
This chain of equivalences will be crucial in the proof of the main theorem, which provides an explicit basis for $VG(\K)$ in terms of the  \emph{$\K$-NBC monomials}
\(
e_N = \prod_{i\in N} e_i
\)
for $N\in NBC(\K)$.

\subsection{Some Commutative Algebra}\label{sec:background:groebner}

This section reviews some commutative algebra material on polynomial rings over $\Z$ and and their quotients.  For more details, see \cite{Adams-Loustaunau,AtiyahMacdonald,eisenbud}.

\subsubsection{Monomial orders}

A polynomial in $\Z[e_1,\ldots,e_n]$ is a sum
\[
f=\sum_{a=(a_1,\ldots,a_n)} c_a \, e_1^{a_1} \cdots e_n^{a_n}
\]
where $c_a\in\Z$. When $c_a \neq 0$, one calls $c_a \, e_1^{a_1} \cdots e_n^{a_n}$ a {\it term} of $f$, and
$e_1^{a_1} \cdots e_n^{a_n}$ a {\it monomial} of $f$.
Define $\deg(f):=\max\{\sum_i a_i: c_a \neq 0\}$
and then the {\it degree-initial form} of $f$ is
\[
\init_{\deg}(f):= \sum_{\substack{a\\ \sum_i a_i =\deg(f)} } c_a \, e_1^{a_1} \cdots e_n^{a_n}.
\]
A {\it monomial ordering} is a total (linear) order $\prec$ {\it well-ordering} on the set of all monomials $m$ in  $\Z[e_1,\ldots,e_n]$ which respects multiplication
in the sense that
$m\prec m'$ implies $m \cdot m''\prec m'\cdot m''$ for all monomials $m, m', m'' \in \Z[e_1,\ldots,e_n]$.
Define the \emph{$\prec$-leading monomial} $\init_\prec(f)$ to be the $\prec$-highest monomial of $f$.
Say that $\prec$ is a {\it degree order} if it is compatible
with $\init_{\deg}$ in the sense that
\(
\init_\prec(f)=\init_\prec(\init_{\deg}(f))
\)
for all $f$; see Sturmfels \cite[Chapter 1]{Sturmfels} for more on
these notions.
Given a collection $\GGG=\{ g_i \}_{i \in I}$ of polynomials,
say that a monomial $m$ is {\it $\init_\prec(\GGG)$-standard} if it is
divisible by none of $\{ \init_\prec(g_i) \}_{i \in I}$.

\subsubsection{Filtrations and Associated Graded Rings}

Let $R$ be a commutative ring with unit. An \emph{(ascending) filtration} of $R$ is a sequence $F_0\subseteq F_1 \subseteq F_2 \subseteq \dots$ of nested $\Z$-submodules of $R$
with the property that if $f\in F_c$ and $g\in F_d$, then $f\cdot g \in F_{c+d}$. In this paper, we consider the \emph{degree filtration} $\{F_d\}_{d \geq 0}$ for quotient rings
$R = \Z[e_1,\dots, e_n]/I$, where $I$ is an ideal of $ \Z[e_1,\dots, e_n]$:  define
$F_d$ to be the image within $R$ of the polynomials in $\Z[e_1,\ldots,e_n]$
having degree at most $d$.
Define the associated graded ring
\[
\grr(R):=\grr_{\mathcal{F}}{(R)} := \bigoplus_{d\geq 0} F_d/F_{d-1}
\]
where we define $F_{-1}:=0$.

Recall that the \emph{rank} of a $\Z$-module $M$ is
\(
\text{rank}_{\Z}(M) := \dim_{\Q} \left( \Q \otimes_\Z M \right),
\)
see \cite[Section 11.6]{eisenbud}.  In the setting of
a degree filtration, each $F_d$ is a finitely generated $\Z$-module,
allowing us to define the Hilbert series of the associated graded ring:
\[
\text{Hilb} (\grr(R),t) : = \sum_{d\geq 0} \text{rank}_{\Z}(F_d/F_{d-1})~t^d.
\]
For example, we will wish to consider the associated graded ring of the Varchenko-Gel'fand ring with its degree filtration $\vgrr(\K):=\grr(VG(\K))$, with Hilbert series
$
\text{Hilb} (\vgrr(\K),t).
$

The proof of Theorem \ref{thm:main theorem} in Section \ref{sec:definition of GGG} uses a certain general lemma, which we state and prove now.
Experts may recognize this lemma as a standard fact from
Gr\"{o}bner basis theory when the polynomial rings are defined over a field, but the modification here relates to polynomial rings over $\Z$;
see Remark \ref{rem:integer-grobner-field} below.

\begin{lemma}\label{lem:integer-grobner}
Let $\preceq$ be a monomial order and assume one has a $\Z$-algebra surjection
$S:=\Z[e_1,\ldots,e_n] \overset{\varphi}{\twoheadrightarrow} R$
in which $R$ is a free $\Z$-module of rank $r$,
and $\GGG=\{g_i\}_{i \in I} \subset S$ has these properties:

\begin{itemize}
\item[(i)]  $\GGG \subset \ker \varphi$.
\item[(ii)]  Each $g_i$ is $\preceq$-monic, meaning $\init_\prec(g_i)$ has coefficient $\pm1$ in $g_i$.
\item[(iii)] The set of $\init_\prec\GGG$-standard monomials $\NNN=\{m_1,\dots,m_t\}$ has cardinality $t\leq r$.
\end{itemize}

Then one has these implications:

\begin{itemize}
\item[(a)] $\ker(\varphi)=(\GGG)$, so that $\varphi$ induces a $\Z$-algebra isomorphism
\(
S/(\GGG) \cong R.
\)

\item[(b)] The cardinality $\#\NNN=t=r$, and $R$ has $\varphi(\NNN)=\{\varphi(m_1), \ldots, \varphi(m_r)\}$ as a $\Z$-basis.
\end{itemize}

If  $\preceq$ is also a degree ordering, then one has two further implications:

\begin{itemize}
\item[(c)]
    The map $S \overset{\psi}{\rightarrow}  \grr(R)$ sending $e_i \mapsto \bar{x}_i$ in $F_1/F_0$ is surjective, with $\ker(\psi)=(\init_{\deg}(\GGG) )$, so that it induces a $\Z$-algebra isomorphism
\[
S /( \init_{\deg}(\GGG) ) \overset{\psi}{\rightarrow}  \grr(R)
\]
\item[(d)] For $d\geq 0$, each $F_d/F_{d-1}$ is a free $\Z$-module on the basis $\{m_i \in \NNN: \deg(m_i)=d\}$,
so that
\[
\Hilb(\grr(R),t)=\sum_{i=1}^r t^{\deg(m_i)}.
\]
\end{itemize}
\end{lemma}

\begin{proof}
Note since each $g_i$ in $\GGG$ is $\prec$-monic,  the usual
{\it multivariate division algorithm}
with respect to $\GGG$ using the order $\prec$ (see Cox, Little and O'Shea \cite[\S 2.3, Theorem 3]{cox-little-oshea}) shows that
every $f$ in $S$ lies in
\(
\Z m_1+\cdots + \Z m_t+(\GGG).
\)
Therefore, if one defines a $\Z$-module map $\Z^t \rightarrow S$ that sends the $i^{th}$ standard basis
element of $\Z^t$ to the $\init_\prec(\GGG)$-standard monomial $m_i$, then the
composite $\Z$-module map $\Z^t \rightarrow S \twoheadrightarrow S/(\GGG)$ is surjective.  This composite is the
map $\alpha$ in this sequence of $\Z$-module surjections/isomorphisms
\begin{equation}
\label{sequence-of-surjections}
\Z^t \overset{\alpha}{\twoheadrightarrow} S/(\GGG) \overset{\beta}{\twoheadrightarrow} S/\ker(\varphi) \cong R \cong \Z^r,
\end{equation}
where $\beta$ comes from our assumption (i) above.
It is well-known (see \cite[Chapter 2, Exercise 12]{AtiyahMacdonald}, for example) that if $M$ is a $\Z$-module with $\text{rank}_{\Z}(M) = r \geq t$, then any surjection $\Z^t\to M$ must in fact be an isomorphism, with $t=r$. It follows that the composite of
all maps in \eqref{sequence-of-surjections} is an isomorphism. Thus $\beta$ and $\beta \circ \alpha$  are isomorphisms,
proving assertions (a) and (b), respectively.

Now assume further that $\prec$ is a degree ordering.  Since
$
\init_\prec(\init_{\deg}(g_i))=\init_\prec(g_i),
$
replacing $\GGG$ with $\init_{\deg}(\GGG)$, we conclude as in the above proof
that the composite map
$$
\Z^r \rightarrow S \twoheadrightarrow S/(\init_{\deg}(\GGG))
$$
is surjective.
The fact that $S \overset{\varphi}{\twoheadrightarrow} R$ is surjective implies
that $S \overset{\psi}{\twoheadrightarrow} \grr(R)$ is also surjective.
Furthermore,  the definitions of $\init_{\deg}(-)$ and $\grr(R)$, together with  $\GGG \subset \ker(S \overset{\varphi}{\rightarrow} R)$, imply
\[
\init_{\deg}(\GGG) \subset \ker(S \overset{\psi}{\twoheadrightarrow} \grr(R)).
\]
Hence we again have a sequence of surjections and isomorphisms:
\begin{equation}
\label{gr-sequence-of-surjections}
\Z^r \overset{\gamma}{\twoheadrightarrow} S/(\init_{\deg}(\GGG)) \overset{\delta}{\twoheadrightarrow} S/\ker(\psi) \cong \grr(R).
\end{equation}
Note that $\rank_\Z (\grr(R)) =\rank_\Z R = r$, since $\rank_\Z(-)$ is additive along short exact sequences and direct sums.
Thus we can again conclude
that the composite of the surjections in \eqref{gr-sequence-of-surjections} is an isomorphism.
Hence $\delta$ is an isomorphism, proving (c).  Then (d) follows from $\delta \circ \gamma$ being an
isomorphism, upon noting that a monomial $m$ in $S$  has $\psi(m)$ lying in $F_d/F_{d-1}$ where $d=\deg(m)$.
\end{proof}

\begin{remark}\label{rem:integer-grobner-field}
Replacing $\Z$ by a field $\kk$, and replacing $\rank_\Z(-)$ with
$\dim_\kk(-)$, the proof of Lemma~\ref{lem:integer-grobner} shows $\GGG$ and $\init_{\deg}(\GGG)$ give Gr\"obner bases for the ideals presenting the rings $R$ and $\grr(R)$.
\end{remark}

\section{Relations among the Heaviside functions and the $\init_\prec(\GGG)$-standard monomials}\label{sec:definition of GGG}

Given a cone $\K$ in an arrangement $\A$, let $\GGG$ be the elements
shown in the second column of the
table in Figure~\ref{chart}.
In this section, we give two propositions regarding $\GGG$, which together prove Theorem \ref{thm:main theorem}.
First we show that the polynomials $\GGG$
lie in the kernel of the map $\varphi: \Z[e_1,\ldots,e_n] \longrightarrow VG(\K)$ which sends the variable $e_i$ to the Heaviside function $x_i$ for each $i\in [n]$.
After that, we fix a monomial order $\prec$ on $\Z[e_1,\dots,e_n]$ whose restriction to the variables is $e_1\prec \cdots\prec e_n$. We will show that the $\K$-NBC monomials are
exactly the $\init_\prec(\GGG)$-standard monomials. In fact, since Equation \eqref{cone-nbc-sets-equinumerous-with-cone-chambers} implies
$
VG(\K) \cong \Z^{\#\C(\K)} = \Z^{\# NBC(\K)},
$
using Lemma \ref{lem:integer-grobner}
it suffices to show that the $\init_\prec(\GGG)$-standard monomials are a subset of the $\K$-NBC monomials.

\begin{prop}\label{prop:containment}
Every polynomial in $\GGG$ lies in $\ker\varphi$.
\end{prop}

\begin{proof}
This holds for {\sf Idempotent} relations $e_i^2-e_i$
since Heaviside functions $x_i$ have $x_i(C) \in \{0,1\}$.
It holds for {\sf Unit} relations $e_i-1$ with $i \in W$,
since then $H_i^+ \supseteq \K$,
so $x_i(C)\equiv 1$ for all $C$ in $\C(\K)$.

To understand the {\sf Circuit} and {\sf Cone Circuit} relations,
note that the existence of a signed circuit $C = (C^+,C^-)$ implies that
these two intersections are empty, and hence contain no chambers:
\[
    \bigcap_{i\in C^+} H_i^+ \cap \bigcap_{j\in C^-} H_j^-
    =\emptyset
    = \bigcap_{i\in C^-} H_i^+\cap \bigcap_{j\in C^+} H_j^-
\]
Consequently, if one writes the {\sf Circuit} relation
as the difference $f_+-f_-$ of these two products
\begin{equation}
    \label{two-sides-of-circuit-relation}
f_+:=e_{C^+} \cdot \displaystyle\prod_{j \in C^-} (e_j-1)
\quad \text{ and } \quad
f_-:=e_{C^-} \displaystyle \prod_{j \in C^+} (e_j-1),
\end{equation}
one finds that both $f_+,f_-$ lie in $\ker \varphi$,
and hence so does the {\sf Circuit} relation $f_+-f_-$.

For a {\sf Cone Circuit} relation, assume without loss of generality that the signed
circuit $C = (C^+,C^-)$ has
$\emptyset \neq W \cap C_+ = W \cap C$.
Then since $\ker\varphi$ contains the product
$f_+$ defined in \eqref{two-sides-of-circuit-relation}
along with the {\sf Unit} relations $e_i-1$ for $i \in W \cap C^+$,
it also contains the {\sf Cone Circuit} relation
$e_{C^+\setminus W} \cdot \prod_{j \in C^-} (e_j-1)$.
\end{proof}

\begin{rem}
Note that if the signed circuit $C=(C^+,C^-)$ has
$W\cap C^+\not=\emptyset$, the element $f_-$ is divisible
by {\sf Unit} relations $e_i-1$ for $i \in W\cap C^+$,
and hence is superfluous in generating $\ker\varphi$.  Similarly,
 if $W\cap C^-\not=\emptyset$, then $f_+$ is a redundant generator.
Combining these: if both
$W\cap C^-\not=\emptyset$ and $W\cap C^+\not=\emptyset$, then both $f_+, f_-$ are redundant and so is the corresponding {\sf Circuit} relation.

Also note, when $H_i\in\A$ is neither one of the chosen set of defining hyperplanes nor
the union $W\cup \{i\}$ contains a signed circuit $C = (C^+,C^-)$. Furthermore, one can show that $i\in\underline{C}$ and that the {\sf Cone-Circuit} relation does not vanish, i.e. one of $e_i - 1$ or $e_i$ is in $\GGG$.
\end{rem}

\begin{prop}\label{prop:vg-assumption-iii}
  The $\init_\prec(\GGG)$-standard monomials are (a subset of the) $\K$-NBC monomials.
\end{prop}

\begin{proof}
Let $m\in \Z[e_1,\dots,e_n]$ be any $\init_\prec(\GGG)$-standard monomial. We show that $m$ is a $\K$-NBC monomial
in several reduction steps.

\bigskip

\noindent
{\sf Reduction 1.}
Since $e_i^2 \in \init_{\prec}(G)$ for $1 \leq i \leq n$,
we may assume that $m = e_N$ for some $N\subseteq\{1,\dots,n\}$.

\bigskip

\noindent
{\sf Reduction 2.}
Since $e_i \in \init_{\prec}(G)$ for $i \in W$, we may assume that $m=e_N$ with $W\cap N = \emptyset$.

\bigskip

\noindent
{\sf Reduction 3.}
We can assume that $m = e_N$ where $N$ contains no broken circuits, i.e. $N \in NBC(\A)$.
To see this, suppose $N$ contains a signed circuit $C$ with $i_0 = \text{min}(\underline{C})$ such that the corresponding broken circuit $\underline{C}\backslash\{i_0\}$ is contained in $N$.

Since $W\cap N = \emptyset$ (from {\sf Reduction 2}) and  $\underline{C}\backslash\{i_0\}\subseteq N$, either $i_0\in W$ or $W\cap \underline{C}$ is empty. We obtain a contradiction in both cases. First, if $W \cap \underline{C}=\{i_0\}$, then
 \[
 N \supseteq \underline{C} \setminus \{i_0\} = \underline{C} \setminus W,
 \]
 forcing $e_N$ to be divisible
 by $e_{\underline{C} \setminus W}$, and contradicting that $e_N$ is
 $\init_\prec(\GGG)$-standard.  On the other hand, if $\#W \cap \underline{C}=\emptyset$, then $e_N$ is divisible by $e_{\underline{C} \setminus \{i_0\}}$  which
 contradicts the assumption that $e_N$ is $\init_\prec(\GGG)$-standard.

\bigskip

\noindent
{\sf Reduction 4.}
Assuming that $m=e_N$ where $N$ is in $NBC(\A)$, we will show that it also lies in
$NBC(\K)$, that is,
 $X:=\cap_{j \in N} H^0_j$ has $\K \cap X \neq \emptyset$.
For the sake of contradiction, assume
 \[
 \K \cap X =
 \bigcap_{i\in W}H_i^+ \cap \bigcap_{j\in N} H^0_j = \emptyset.
 \]
It is easy to see that there is a choice of signs $\varepsilon\in \{+,-\}^N$
for which
\begin{align}\label{eqn:empty-intersection}
    \bigcap_{i\in W}H_i^+ \cap \bigcap_{j\in N} H^{\varepsilon_j}_j = \emptyset
\end{align}
(see Observation \ref{obs:covectors}, below, for details).
Translating Equation \eqref{eqn:empty-intersection} into the language of oriented matroids, we have that
there is no covector $F = (F^+,F^-)$ of the matroid on $E = W\cup N$ with
\[
F_j
=
\begin{cases}
+ & \text{if } j\in W\\
\varepsilon_j & \text{if } j\in N.
\end{cases}
\]
From Observation \ref{obs:gordan} (below) or, equivalently, Gordan's Theorem \cite{gordan}, the fact that no such $F$ exists, means that there does exist a (nonzero) signed dependence $D=(D^+,D^-) \in \mathfrak{D}$ with
$\underline{D} \subseteq W\cup N$ and
$(W\cap \underline{D})\subseteq D^+$.

Recall from \cite[Theorem 3.7.5]{BjornerEtAl} that every signed dependence $D$ is a composition of circuits and that at least one of these circuits\footnote{There is always a choice composition for which \emph{every} circuit conforms to $D$ but we only need one conformal circuit, see \cite[Proposition 3.7.2]{BjornerEtAl}.} must conform to $D$. Let $C$ be such a circuit conforming to $D$.
Then $C$ has
has $\underline{C}\subseteq W\cup N$ and
$(W\cap \underline{C})\subseteq C^+$. Since $N$ is an NBC set, the corresponding collection of vectors $\{v_i\}_{i\in N}$ is independent and we can assume that $W\cap \underline{C}$ is nonempty.
Thus its {\sf Cone-Circuit} relation has initial form $e_{\underline{C}\backslash W}$ dividing $e_N$, contradicting
$e_N$ being  $\init_\prec(\GGG)$-standard.
\end{proof}

Combining the preceding proposition with Lemma \ref{lem:integer-grobner} gives a proof of Theorem \ref{thm:main theorem}. For completeness, we now state two observations about oriented matroids, which were used in the preceding proof. The first concerns the geometric interpretation of covectors as faces of hyperplane arrangements and the second observation connects the non-existence of a covectors to the existence of a vector.

\begin{observation}\label{obs:covectors}
If
\(
 \K \cap X =
 \bigcap_{i\in W}H_i^+ \cap \bigcap_{j\in N} H^0_j = \emptyset
\)
then there is some choice of signs $(\varepsilon_j)_{j \in N}$ in $\{+,-\}^N$ such that
\[
 \K \cap X =
 \bigcap_{i\in W}H_i^+ \cap \bigcap_{j\in N} H^{\varepsilon_j}_j = \emptyset.
\]
In particular, the signed set $F^{\varepsilon} = ((F^{\varepsilon})^+, (F^{\varepsilon})^-)$
with
\[
F_i^{\varepsilon} =
\begin{cases}
+ & \text{if } i \in W\\
\varepsilon_i & \text{if } i \in N.
\end{cases}
\]
is not a covector of the oriented matroid on ground set $E = W\cup N$.
\end{observation}

Another way to phrase $X\cap\K = \emptyset$ is: there are no covectors $F$ having $F_i=+$ for all $i \in W$
and $F_j=0$ for all $j \in N$. With that in mind, the observation holds because if no such choice of signs $(\varepsilon_j)_{j \in N}$ existed,
one would obtain a family of covectors
$\{ F^{\varepsilon} \}_{\varepsilon \in \{+,-\}^N}$ to which one could repeatedly apply the elimination axiom ${\sf V3}$ and reach such a covector $F$ having $F_j=0$ for $j \in N$.

\begin{observation}[{{\cite[Section 6.3]{ziegler}}}]\label{obs:gordan}
From the definition of an oriented matroid dual, we know that every $F\in \mathfrak{D}^\ast$ is orthogonal to \emph{every} vector $D\in\mathfrak{D}$. In particular if there is a signed set $F$ on $[n]$ that is \emph{not} in $\mathfrak{D}^\ast$, then there is some $D\in\mathfrak{D}$ such that $\{F_e\cdot D_e \mid e\in E\}$ contains exactly one of $+$ or $-$.
\end{observation}

\begin{remark}
The crux of the preceding proof only uses
statements about vectors and covectors valid for oriented matroids.
Hence our results remain valid in that setting, as in the generalization by Gel'fand and Rybnikov \cite{GelfandRybnikov}
of the work by Gel'fand and Varchenko \cite{VarchenkoGelfand} to oriented matroids.
\end{remark}

\section{Proof of Theorem \ref{thm:koszul}}\label{sec:supersolvability}
 In this section we prove Theorem \ref{thm:koszul}, asserting that when one works over a field $\kk$, the associated graded ring
 $\vgrr_\kk(\A)$ is Koszul whenever $\A$ is a supersolvable arrangement.
 Some of the most well-studied hyperplane arrangements are supersolvable,
 and supersolvable arrangements are interesting, for example, because
 their Poincar\'e polynomial
 $\Poin(\A,t)$ factors into linear factors\footnote{This does not hold for cones of supersolvable arrangements, see \cite[Remark 5.6]{DorpalenKimReiner}.}
 \cite[Corollary 4.9]{stanley-hyperplanes}.
 Koszulity, on the other hand, is also interesting from an algebraic perspective. Koszul algebras come equipped with a natural Koszul dual quadratic algebra $A^{!}$,
 and the relationship between $A, A^{!}$ has implications for the coefficients of the Hilbert series of $A$.

Let $\kk$ be a field. Recall that $VG_{\kk}(\K)$ is the collection of maps $\{f: \C(\K) \to \kk\}$ with pointwise addition and multiplication. Theorem \ref{thm:main theorem} extends without modification to $VG_{\kk}(\K)$, and in fact some of the proofs are easier since $\GGG$ forms a Gr\"{o}bner basis (see Remark \ref{rem:integer-grobner-field}).

Before beginning the proof of Theorem \ref{thm:koszul}, we remind the reader of some standard results relating to supersolvable lattices and Koszul algebras. For a more detailed reference on Koszul algebras, we point the reader toward \cite{froberg97}.
Let $R$ be a \emph{commutative standard graded $\kk$-algebra}, i.e. $R\cong \kk[e_1,\dots, e_n]/I$ where $I$ is a homogeneous ideal and each $e_i$ has degree exactly $1$.
Suppose
\[
F_\bullet: \cdots \overset{\varphi_3}{\longrightarrow} R^{\beta_2} \overset{\varphi_2}{\longrightarrow} R^{\beta_1}
\overset{\varphi_1}{\longrightarrow} R
\longrightarrow \kk
\]
is a \emph{minimal free resolution} of the $R$-module $\kk = R/R_{+}$ where $R_+$ is the maximal homogeneous ideal, consisting of all elements of positive degree.
For details on free resolutions, see \cite{eisenbud}. We say $R$ is \emph{Koszul} if the nonzero entries of each $\varphi_i$ matrix are homogeneous of degree 1.
Say that an ideal is \emph{monomial} if it is generated by monomials. A monomial ideal is \emph{$\GGG$-quadratic} if it has a Gr\"{o}bner basis of monomials of degree two.  It is well-known (see \cite{conca},\cite[Chapter 15]{eisenbud}, \cite[Section 4]{froberg97}, for example) that:

\begin{prop}\label{prop:ggg-quad-means-koszul}
If $I$ a homogeneous ideal in $\kk[e_1,\dots, e_n]$
is generated by a Gr\"obner basis $\GGG$ consisting of quadratic
elements for some monomial order $\prec$,
then $R=\kk[e_1,\dots, e_n]/I$ is Koszul.
\end{prop}

We are now prepared to define a supersolvable arrangement.

 \begin{defn}[{\cite[Definition 1.1]{stanley-lattices}, \cite[Definition 4.13]{stanley-hyperplanes}}]
A lattice $L$ is \emph{supersolvable} if there exists a maximal chain $\Delta$ satisfying: for every chain $K$ of $L$, the sublattice generated by $\Delta$ and $K$ is distributive.
An arrangement $\A$ is \emph{supersolvable} if $\LL(\A)$ is supersolvable.
\end{defn}

 The following result and its proof are analogous to
 a result of Peeva \cite[Theorem 4.3]{peeva}.

 \medskip

 \noindent\textbf{Theorem \ref{thm:koszul}}.
  If $\A$ is a supersolvable arrangement, then $\vgrr_{\kk}(\A)$ is Koszul.

 \begin{proof}
 A theorem of Bj\"{o}rner and Ziegler \cite[Theorem 2.8]{bjorner-ziegler} tells us that when $\A$ is a supersolvable arrangement, one can
 choose a linear ordering of the hyperplanes $H_1,\ldots,H_n$
 such that every broken circuit $\underline{C} \setminus \{i_0\}$
 contains some broken circuit of
 size two. Choose $\prec$ a degree monomial order on $\Z[e_1,\ldots,e_n]$ which restricts to the same
linear order on the variables $e_1 \prec \cdots \prec e_n$.

We wish to use the presentation $\vgrr_\kk(\A)=\kk[e_1,\ldots,e_n]/I$
 where $I=(\init_{\deg}(\GGG))$ that comes from Theorem~\ref{thm:main theorem}. From Remark \ref{rem:integer-grobner-field}, the generators $\init_{\deg}(\GGG)$ form a Gr\"{o}bner basis for $I\subseteq \kk[e_1,\ldots,e_n]$.
 Because $\A$ is a full arrangement, not a cone,
 $\init_{\deg}(\GGG)$ will contain only $\init_{\deg}(g)$ for
 {\sf Idempotent} and {\sf Circuit} relations $g$.
 The {\sf Idempotent} relations correspond to
 generators $\init_{\deg}(g)=e_i^2$ that are all quadratic.

 Each {\sf Circuit} relation corresponds to
 a generator $\init_{\deg}(g)$ which may not be
 quadratic:  its  degree is the size of the
 broken circuit $C \setminus \{i_0\}$, with
 $\init_{\prec}(g)=e_{C \setminus \{i_0\}}$. Since each is a squarefree monomial, it suffices to consider the monomials who indexing set is minimal under inclusion.
The Bj\"{o}rner-Ziegler result \cite[Theorem 2.8]{bjorner-ziegler}  implies that the minimal (under inclusion) broken circuits all have cardinality $2$.
From Proposition \ref{prop:ggg-quad-means-koszul}, it follows that $\vgrr_{\kk}(\A)$ is Koszul.
\end{proof}

One might ask if Theorem \ref{thm:koszul} has a cone analogue. Sadly there are cones $\K$ of supersolvable arrangements whose $\vgrr_{\kk}(\K)$ are not Koszul as we demonstrate by example. A particularly well-studied family of supersolvable arrangements are the \emph{Type A reflection arrangements} or \emph{braid arrangements};
see \cite[Cor. 4.10, Example 4.11(c)]{stanley-hyperplanes}. The braid arrangement $A_{n-1}$, consists of the
$\binom{n}{2}$ hyperplanes
\[
H_{ij} = \{ (x_1,x_2,\dots,x_n)\in \R^n  \mid x_i - x_j = 0\}
\]
for each pair $\{i,j\}$.

We wish to exhibit a cone $\K$ inside a braid arrangement
for which $\vgrr(\K)$ is not Koszul.
One way to prove something is \emph{not} Koszul uses
the following.

\begin{thm}[{{\cite[Section 4]{froberg75}}}]
Let $A$ be a Koszul algebra. Then there is another algebra $A^!$, the quadratic dual of $A$, whose Hilbert series is
$
\Hilb(A^!, t) = 1/\Hilb(A,-t).
$
In particular, if $A$ is Koszul, then $1/\Hilb(A,-t)$ has positive coefficients when considered as a power series in $\Z[t]]$.
\end{thm}

\begin{example}
\noindent The cone of $A_5$ given by
\[
\K = \{\mathbf{x}\in\R^6 \mid x_1\leq x_2,~x_3\leq x_4,~x_5\leq x_6\}
\]
does not yield a Koszul $\vgrr_{\kk}(\K)$.
The Hilbert series of $\vgrr_{\kk}(\K)$ is
\[
\Hilb(\vgrr_{\kk}(\K),t) = 1+ 12t+43t^3 +30t^3 + 4t^4.
\]
The first few terms of $1/\Hilb(\vgrr_{\kk}(\K),-t)$ are
\begin{align*}
   \frac{1}{1-12t+43t^3-30t^3 + 4t^4}
   & =  1 + 12t + 101t^2 + 725t^3 + 4725t^4
   + 28464t^5
   + 159769t^6\\
   & +  832122t^7 + 3950417t^8 + 16302972t^9
   + 50092317 t^{10}\\
   & + 15264030 t^{11}  - 1497513779t^{12}
+ \cdots
\end{align*}
The coefficient of $t^{12}$ is negative, meaning that $1/\Hilb(\vgrr_{\kk}(\K),-t)$ is not the Hilbert series of a ring.
\end{example}

This counterexample begs the following question:

\begin{question}
Is there some simple, combinatorial condition on $\LI(\K)$ for a cone $\K$, in the spirit of supersolvability for the full lattice $\LL(\A)$ of the arrangement, that implies Koszulity of $\vgrr_{\kk}(\K)$?
\end{question}

Even in the arrangement case, the connection between supersolvable arrangements and Koszul $\vgrr_{\kk}(\A)$ remains opaque. Fr\"{o}berg tells us that the converse to Proposition \ref{prop:ggg-quad-means-koszul} is false in general \cite[Note on p.39]{froberg75}, but one might ask:
 if $\vgrr_{\kk}(\A)$ is Koszul, is $\A$ supersolvable? The analogous question is famously open for Orlik-Solomon algebras \cite[Example 4.5]{peeva}, \cite[Example 6.22]{yuzvinsky}.

\section*{Acknowledgements}

The author thanks Victor Reiner for motivating discussions, Franco Saliola for enlightening conversations about SAGE \cite{sagemath}, and the anonymous reviewer for many helpful comments.
She also thanks Sarah Brauner, Darij Grinberg, and Trevor Karn for pointing out typos/giving feedback on drafts of this paper.



\bibliography{bibliography}{}
\bibliographystyle{plain}

\end{document}